\documentclass[centering,11pt,reqno]{amsart}  

\usepackage[utf8]{inputenc}
\usepackage[T1]{fontenc}
\usepackage{amsmath,amsthm}
\usepackage{amsfonts,amssymb}
\usepackage[mathscr]{eucal}
\usepackage{url}
\usepackage{paralist}
\usepackage[colorlinks,cite color=blue,pagebackref=true,pdftex]{hyperref}
\usepackage[margin=2.5cm]{geometry}
\usepackage{tikz}
\usetikzlibrary{calc,shapes}
\usepackage{mathtools}
\usepackage{blkarray}
\usepackage[capitalize]{cleveref}
 \usepackage[foot]{amsaddr}
 \usepackage{colonequals}

\newtheorem{theorem}{Theorem}[section]
\newtheorem*{lemma*}{Lemma}
\newtheorem*{prop*}{Proposition}
\newtheorem{lem}[theorem]{Lemma}
\newtheorem{cor}[theorem]{Corollary}
\newtheorem{prop}[theorem]{Proposition}
\newtheorem{conjecture}[theorem]{Conjecture}

\newtheorem{op}[theorem]{Open Problem}

\theoremstyle{definition}
\newtheorem{example}[theorem]{Example}

\newtheorem{definition}[theorem]{Definition}
\newtheorem{remark}[theorem]{Remark}

\newcommand\note[1]{\textbf{#1}}

\renewcommand{\epsilon}{\varepsilon}
\newcommand{\R}{\mathbb{R}}
\newcommand{\calC}{\mathcal{C}}

\newcommand{\tabfunc}{\mathcal T}
\definecolor{alexmcolor}{RGB}{9,6,250}
\definecolor{amandacolor}{RGB}{50,150,50}

\DeclareMathOperator{\argmax}{argmax}
\DeclareMathOperator{\SYT}{SYT}
\DeclareMathOperator{\rSYT}{rSYT}
\DeclareMathOperator{\rSCYT}{rSCYT}
\def\R{\mathbb{R}}

\newcommand{\calA}{\mathcal{A}}

\newcommand{\calP}{\mathcal{P}}

\date{\today}



\title{Realizable Standard Young Tableaux}

\author[Araujo]{Igor Araujo\textsuperscript{*}}
\address{\textsuperscript{*}Department of Mathematics, University of Illinois at Urbana-Champaign, Urbana, IL, USA 61801}

\author[Black]{Alexander E. Black\textsuperscript{\dag}}
\address{\textsuperscript{\dag}Department of Mathematics, Bowdoin College, Brunswick, ME 04011}

\author[Burcroff]{Amanda Burcroff\textsuperscript{\ddag}}
\address{\textsuperscript{\ddag}Department of Mathematics, MIT, Cambridge, MA, USA 02139}

\author[Gao]{Yibo Gao\textsuperscript{\S}}
\address{\textsuperscript{\S}Department of Mathematics, Peking University, Beijing, China 100871}

\email{\parbox[t]{\linewidth}{igoraa2@illinois.edu, a.black@bowdoin.edu, amandabu@mit.edu,\\
gaoyibo@bicmr.pku.edu.cn, rkrueger@andrew.cmu.edu, \and alexmcd@uoregon.edu}}

\author[Krueger]{Robert A. Krueger\textsuperscript{+}}
\address{\textsuperscript{+}Department of Mathematics, Carnegie Mellon University, Pittsburgh, PA, USA 15213}

\author[McDonough]{Alex McDonough\textsuperscript{$\circ$}}
\address{\textsuperscript{$\circ$}Department of Mathematics, University of Oregon, Eugene, OR, USA 97403}

\begin{document}

\maketitle

\begin{abstract}
Given two vectors $u$ and $v$, their outer sum is given by the matrix $A$ with entries $A_{ij} = u_{i} + v_{j}$. If the entries of $u$ and $v$ are increasing and sufficiently generic, the total ordering of the entries of the matrix is a standard Young tableau of rectangular shape. We call rectangular standard Young tableaux arising in this way \emph{realizable}. The set of realizable tableaux was defined by Mallows and Vanderbei for studying a deconvolution algorithm, but we show they have appeared in many other contexts including sorting algorithms, quantum computing, random sorting networks, reflection arrangements, fiber polytopes, and Goodman and Pollack's theory of allowable sequences. In our work, we prove tight bounds on the asymptotic number of realizable rectangular tableaux. We also derive tight asymptotics for the number of realizable allowable sequences, which are in bijection with realizable staircase-shaped standard Young tableaux with the notion of realizability coming from the theory of sorting networks. 
\end{abstract}

\section{Introduction}
Motivated by the study of a deconvolution algorithm in~\cite{MallowsVanderbei}, Mallows and Vanderbei defined a novel subset of rectangular standard Young tableaux. We say a vector $x \in \R^k$ is \emph{increasing} if $x_{1} < x_{2} < \dots < x_{k}$. Given two increasing vectors $x \in \R^m$ and $y \in \R^n$, we define their \emph{outer sum} as the matrix $x \circ y \in \R^{m \times n}$ with entries $(x \circ y)_{ij} = x_{i} + y_{j}$. Since $x$ and $y$ are increasing, the entries of $x \circ y$ must be increasing from left to right and top to bottom. For a generic choice of $x$ and $y$, all entries of $x \circ y$ are distinct. In this case, the total order on $[m] \times [n]$ given by the entries of $x\circ y$ must correspond to a standard Young tableau of rectangular shape $m \times n$, which we denote $\tabfunc(x \circ y)$. We call a rectangular standard Young tableau, or rectangular tableau for short, \emph{realizable} if it is equal to $\tabfunc(x \circ y)$ for some $x$ and $y$. See Figure \ref{fig:firstexample} for an example. In what follows, we denote the set of all $m \times n$ rectangular tableaux by $\SYT(m,n)$, and the set of realizable rectangular tableaux by $\rSYT(m,n)$. 

\begin{figure}
\[\begin{blockarray}{cccccc}
& 0 & 1 & 5 & 15 & 16\\
\begin{block}{c(ccccc)}
0 & 0 & 1 & 5 & 15 & 16\\
2 & 2 & 3 & 7 & 17 & 18\\
9 & 9 & 10 & 14 & 24 & 25\\
\end{block} 
\end{blockarray}
\hspace{1 cm}\pmb{\Rightarrow}\hspace{1 cm}
\begin{blockarray}{ccccc}
&&&&\\
\begin{block}{|c|c|c|c|c|}
\hline 1 & 2 & 5 & 10 & 11\\
\hline3 & 4 & 6 & 12 & 13\\
\hline7 & 8 & 9 & 14 & 15\\
\hline
\end{block} 
\end{blockarray}
\]
 \caption{Above are the matrix given by $x\circ y$ and the rectangular tableau $\tabfunc(x \circ y)$ for $x~=~(0,2,9)$ and $y~=~(0,1,5,15,16)$.}\label{fig:firstexample}
\end{figure}

As a first step for studying realizable rectangular tableaux, Mallows and Vanderbei showed that all $2 \times n$ tableaux are realizable and gave examples of $m \times n$ tableaux for small choices of $m$ and $n$ that are not realizable. They concluded that the problem of studying realizable tableaux is difficult and required some new perspective. The goal of our work is to asymptotically enumerate these tableaux and describe many other contexts in which they arise including sorting algorithms, quantum computing, random sorting networks, reflection arrangements, fiber polytopes, and Goodman and Pollack's theory of allowable sequences. 

The key starting point for our analysis comes from relating realizable rectangular tableaux to regions in a hyperplane arrangement, as observed by the second author and Sanyal in~\cite{FlagPoly}.  Define the $\rSYT(m,n)$-arrangement $H$ to be the union of $\{(x,y) \in \R^{m} \times \R^{n}: x_{i} + y_{k} = x_{j} + y_{\ell}\}$ over all $i, j  \in [m]$ and $k, \ell \in [n]$. Then, the set of regions of $H$ contained in the cone $\{(x,y): x_{i} < x_{i+1}, y_{j} < y_{j+1} \text{ for all } i \in [m], j \in [n]\}$ is in bijection with $\rSYT(m,n)$. Thus, $|\rSYT(m,n)| = \frac{|r(H)|}{m!n!}$, where $r(H)$ denotes the set of regions of $H$.

This hyperplane arrangement and the corresponding observation also appeared independently due to Klyachko in the context of quantum computing for the $N$-representability problem for qubits (see Remark $3.4.2$ of~\cite{Klyachko}). In that context, Klyachko called realizable rectangular tableaux \emph{cubicles}. The case of cubicles is a special case of a larger story of arrangements associated to fundamental particles in physics \cite{ConvexRepresentability, DensityFoundations, LineupPolys}. This arrangement is also of independent interest in geometric combinatorics for studying generalizations of generalized permutahedra to allow for edge directions of the form $e_{i} -e_{j} + e_{k} - e_{\ell}$, which was one of the primary motivations for the introduction of nested braid fans in \cite{nestedbraidfans}. This geometric perspective is key for proving our first two main theorems, which provide asymptotic bounds on the number of realizable rectangular tableaux.

\begin{theorem}\label{thm::upper and lower}
For all $m, n \in \mathbb{N}$ such that $m,n \geq 3$, we have
\[ |\rSYT(m-1,n)| \left( \binom{n}{2} + 1 \right) \leq |\rSYT(m,n)| \leq \left(\frac{(m^2-m+1)e^2}{2}\right)^{m+n} n \left(\frac{n}{m}\right)^{m} .\]
In particular, we have $|\rSYT(n,n)| = n^{4n+o(n)}$.
\end{theorem}

When $m$ is constant, \cref{thm::upper and lower} shows that $|\rSYT(m,n)|$ is exponential in $n$ but is not precise enough to determine the base of the exponential.

Our upper bound arguments generalize an argument of Fredman in~\cite{FredmanUpperBound} from complexity theory. In the algorithmic theory of sorting, a longstanding open problem is to find the time complexity of sorting a list of the form $X + Y = \{x+y: x \in X,~ y \in Y\}$ for two lists $X, Y \subseteq \R$ with $|X| = |Y| = n$. Fredman~\cite{FredmanUpperBound} studied this problem nearly fifty years ago, and in doing so proved an upper bound of $n^{8n + o(n)}$ on the number of orderings of a list of the form $X + Y$, with some more recent work in \cite{lowerboundsumsort, ksumdecisiontrees}. If the two lists $X$ and $Y$ are ordered such that they are increasing, then the set of possible total orderings is the set of realizable rectangular tableaux. Through a more careful analysis of the asymptotics, \cref{thm::upper and lower} improves Fredman's estimate, providing the correct number of orderings of $X+Y$ up to the $o(n)$ term in the exponent.

By the hook length formula, we may establish the following corollary in contrast to the case $m = 2$.

\begin{cor} \label{cor::prob_to_zero}
For any sequence $(m_n)_{n \in \mathbb{N}}$ such that $3 \leq m_n \leq \max(n,3)$ for all $n \in \mathbb{N}$, the probability that a uniformly randomly chosen $m_n \times n$ tableau is realizable tends to $0$ as $n \to \infty$.
\end{cor}

Next, we move to staircase-shaped standard Young tableaux, or staircase tableaux for short, where there is an analogous notion of realizability coming from the theory of sorting networks. A sorting network is any way of sorting the identity permutation to the order reversing permutation using adjacent transpositions, where at each step, we increase the number of inversions by applying a single adjacent transposition. Sorting networks are in bijection with staircase tableaux, monotone paths on the permutahedron, and maximal chains in the weak Bruhat order~\cite{RandSortNets}. Sorting networks also have a notion of realizability called being \emph{geometrically realizable} or \emph{stretchable}; see~\cite{PatternTheorem} or Section \ref{sec:staircase}. A \emph{realizable staircase tableau} is one whose corresponding sorting network is geometrically realizable, and we denote the set of all realizable staircase tableaux by $\rSYT(n)$. In~\cite{PatternTheorem}, Angel, Gorin, and Holroyd showed that the probability that a random sorting network is geometrically realizable tends to $0$, which informs the story of limit shapes of random sorting networks (see~\cite{RandSortNets, ArchLimit} for details). 

The analogy between the notions of realizability for staircase tableaux and rectangular tableaux may be understood in terms of paths on polytopes and the fiber polytope construction. Given a polytope $P \subseteq \R^{n}$ with vertices $V(P)$ and $w \in \R^{n} \setminus \mathbf{0}$, a \emph{monotone path} is a path on the graph of a polytope (i.e., its set of vertices and edges) from a vertex $x$ that minimizes of $w^{T}x$ to one that maximizes $w^{T}x$ such that each step of the path increases $w^{T} x$. A monotone path $\ell$ is called \emph{coherent} with respect to $w$ if there exists some $c \in \R^{n} \setminus \mathbf{0}$ such that the set of vertices $V(\ell)$ in $\ell$ is exactly 
\[\{v \in V(P): v = \argmax_{x \in P}(c + \lambda w)^{T} x \text{ for some } \lambda \in \R\}.\] 
See~\cite{CellStringsOnPolytopes} for more details. In particular, these correspond to vertices of the monotone path polytope introduced in \cite{FiberPoly} as a special case of the fiber polytope construction.

For rectangular tableaux, the second author and Sanyal  \cite{FlagPoly} showed that monotone paths on the permutahedron $\Pi_{m + n}$ with respect to the vector $e_{[m]} = \sum_{i \in [m]} e_{i}$ are in bijection with $S_{m} \times S_{n} \times \SYT(m,n)$, where $S_{n}$ denotes the symmetric group on $n$ elements and the permutahedron $\Pi_n$ is the convex hull of the points in $\R^n$ whose coordinates are a permutation of $\{1,\dots,n\}$. Furthermore, coherent monotone paths with respect to $e_{[m]}$ are in bijection with $S_{m} \times S_{n} \times \rSYT(m,n)$.  For staircase tableaux, there is a similar relationship.  Monotone paths on the permutahedron $\Pi_{n}$ with respect to any generic increasing $w \in \R^{n}$ are in bijection with standard Young tableaux of staircase shape. Furthermore, realizable staircase tableaux are in bijection with monotone paths such that there exists some $w$ for which they are coherent.  The first bijection is the Edelman-Greene bijection \cite{EdelmanGreene}, and the connection to coherence comes from~\cite{sweeppolys}. Hence, there is a fundamental geometric relationship between these two notions of realizability. In particular, characterizing these two types of realizable tableaux corresponds exactly to the basic question of understanding which regions a line can intersect in the braid arrangement.

Our main contribution concerning staircase tableaux is providing tight asymptotics for the number of realizable staircase tableaux, and hence the number of geometrically realizable sorting networks.  Sorting networks are also equivalent to allowable sequences in the language of Goodman and Pollack~\cite{AllowableSeqSurvey}. They called an allowable sequence realizable if its corresponding sorting network is geometrically realizable. Hence, our result is an asymptotic enumeration of realizable allowable sequences. Goodman and Pollack showed in~\cite{GoodmanPollackBounds} that the number of allowable sequences, or equivalently combinatorial types of point configurations of $n$ points in $\R^{2}$, is at most $n^{8n}$.   Our result is an improvement upon theirs in that we provide tight asymptotics. For the upper bound, we use a similar approach to theirs with a more careful analysis, as with the bound we derive from Fredman's work for rectangular tableaux. For the lower bound, we directly generalize their argument for estimating the number of order types of labeled point configurations. 

\begin{theorem}
\label{thm:staircasebounds}
The number of realizable standard Young tableaux of staircase shape $(n, n-1, \dots, 2,1)$ is $n^{5n+o(n)}$. 
\end{theorem} 

In~\cite{PatternTheorem}, Angel, Gorin, and Holroyd showed that the probability of a random sorting network being realizable tends to $0$ as $n \to \infty$ at a rate of at most $e^{-Cn}$ and suggested that this can be improved to $e^{-Cn^{2}}$, for some constant $C$. In fact, up to the recontextualization that we provide here and an application of the well-known hook length formula, their result (\cite[Theorem 1.3]{PatternTheorem}) and the improved rate can both be proven using the work Goodman and Pollack from over 20 years prior \cite{GoodmanPollackBounds} (see \cref{rem:realizablestaircase}). 

The paper is organized as follows. In \cref{sec:rect}, we discuss upper and lower bounds on the number of rectangular realizable standard Young tableaux. In \cref{sec:staircase}, we provide asymptotics for the number of realizable standard Young tableaux of staircase shape. Finally, in \cref{sec:FurDir}, we supply many open problems and further directions for where to go next with realizable tableaux.

\section{Rectangular Tableaux} \label{sec:rect}
In this section, we study the number of realizable rectangular standard Young tableaux.  We place upper bounds by bounding the number of regions in a related hyperplane arrangement, and our lower bounds are obtained by considering iterated row extensions of realizable tableaux.  In the case of fixed height and increasing width, as well as the case of simultaneously increasing height and width, we show that our bounds are asymptotically tight.

\subsection{Obstructions to Realizability}

The smallest nontrivial case to consider are the $2 \times n$ standard Young tableaux. Here, Mallows and Vanderbei showed that there are no obstructions to realizability, i.e., all $2 \times n$ standard Young tableaux are realizable \cite[Theorem 7]{MallowsVanderbei}. In particular, the total number of (realizable) $2 \times n$ standard Young tableaux is given by the $n^{th}$ Catalan number.

For larger rectangular tableaux, realizability of rectangular tableaux is a rather delicate property, and the presence of certain substructures can prevent a tableau from being realizable. Mallows and Vanderbei identified a few such substructures, which they called \emph{taboo configurations}~\cite{MallowsVanderbei}. In the following proposition, we give a general description of a class of obstructions to realizability. 

\begin{prop}\label{prop:tabooconfig}
Let  $M = \begin{bmatrix}M_x\end{bmatrix}_{x \in [m]\times[n]}$ be an $m \times n$ tableau.  If there exist two disjoint sets of equal size $A = \{a_1,\dots,a_\ell\},B=\{b_1,\dots,b_\ell\} \subseteq [m]\times[n]$ such that
\begin{enumerate}[(i)]
\item\label{cond::preal1} $M_{a_k} < M_{b_k}$ for each $k \in \{1,\dots,\ell\}$,
\item\label{cond::preal2} $\left|A \cap \left(\{i\} \times[n]\right)\right| = \left|B \cap \left(\{i\} \times[n]\right)\right|$ for each $i \in [m]$, and
\item\label{cond::preal3} $\left|A \cap \left([m]\times\{j\}\right)\right| = \left|B \cap \left([m]\times\{j\}\right)\right|$ for each $j \in [n]$,
\end{enumerate}
then $M$ is not realizable.
\end{prop}
\begin{proof}
 Suppose for the sake of contradiction that $M$ is realizable. Then there exist $c \in \mathbb{R}^{m}$ and $w \in \mathbb{R}^{n}$ such that the ordering of the entries of the outer sum $C=c \circ w$ is given by $M$.
 
 Consider the sums $S_A = \sum_{k = 1}^\ell C_{a_k}$ and $S_B = \sum_{k = 1}^\ell C_{b_k}$ of the entries indexed by each of the sets $A$ and $B$, respectively. Condition (\ref{cond::real1}) implies that $S_A < S_B$. 
 
Observe that, by collecting coefficients of each $c_{i}$ and $w_{j}$, we have 
\[S_{A} = \sum_{k=1}^{\ell} C_{a_{k}} =  \sum_{i=1}^{m} |A \cap (\{i\} \times [n]\})|c_{i} + \sum_{j=1}^{n} |A \cap [m] \times \{j\}| w_{j}.\]
The same expression holds for $S_{B}$, so conditions (\ref{cond::real2}) and (\ref{cond::real3}) imply that $S_A = S_B$, a contradiction. Therefore the tableau is not realizable. 
\end{proof}

We conjecture that the presence of the substructures described in Proposition \ref{prop:tabooconfig} is the only obstruction to realizability; see Section \ref{sec:FurDir} for further discussion.

\begin{example}
The $3 \times 3$ tableau 
\[
\begin{blockarray}{ccc}
&&\\
\begin{block}{|c|c|c|}
\hline 1 & \textcolor{red}2 & \textcolor{blue}6\\
\hline\textcolor{blue}3 & 5 & \textcolor{red}7\\
\hline\textcolor{red}4 & \textcolor{blue}8 & 9\\
\hline
\end{block} 
\end{blockarray}
\]
is not realizable.  This follows from Proposition \ref{prop:tabooconfig} by choosing $\textcolor{red}A = \{(1,2),(2,3),(3,1)\}$ and $\textcolor{blue}B = \{(2,1),(3,2),(1,3)\}$.  The entries corresponding to $A$ and $B$ are colored red and blue, respectively. 
\end{example}

\subsection{An upper bound via hyperplane arrangements}\label{subsec:upper-bound-rect}

We consider the following set of hyperplanes in $\mathbb{R}^{m+n}$:
\[ \mathcal{C}_{n,m} = \left\{ x_i + y_j = x_k + y_\ell : i,k \in [n],~ j,\ell \in [m], \text{ and $i\neq k$ or $j \neq \ell$} \right\} ,\]
where the coordinates of a point in $\mathbb{R}^{m+n}$ are denoted by $x_1, \dots, x_m, y_1, \dots, y_n$. Let
\[ \calP_{n,m} = \{(x_1,\dots,x_m,y_1,\dots,y_n) \in \R^{m+n} : x_1 < x_2 < \cdots < x_m,~ y_1 < y_2 < \cdots < y_n\} ,\]
and observe that the realizable $m \times n$ standard Young tableaux are in bijective correspondence with the regions inside $\calP_{n,m}$ after partitioning with the hyperplanes in $\calC_{n,m}$.

Note that $\calC_{n,m}$ is symmetric with respect to the action of $S_m$ on the first $m$ coordinates and $S_n$ on the last $n$ coordinates. Hence, the number of regions $\calC_{n,m}$ splits $\mathbb{R}^{m+n}$ into is $m!n!$ times the number of regions $\calC_{n,m}$ splits $P_{n,m}$ into.  In particular, we rely on the following relation: 

\begin{equation}
\label{eqn:regions}
|\text{rSYT}(m,n)| = \frac{r(\calC_{n,m})}{m!n!}, 
\end{equation}
where $r(\mathcal{H})$ denotes the number of regions of a hyperplane arrangement $\mathcal{H}$.  

We are now ready to prove the upper bound of \cref{thm::upper and lower}, while the lower bound is proved in Subsection~\ref{subsec:single-row-extensions}.

\begin{proof}[Proof of Upper Bound in \cref{thm::upper and lower}]
By the previous discussion, the number of regions of the arrangement $\calC_{n,m}$ contained in the region $\calP_{n,m}$ is equal to the number of $m \times n$ realizable standard Young tableaux.  The number of hyperplanes in the arrangement $\calC_{n,m}$ is precisely $2\binom{m}{2}\binom{n}{2}+\binom{m}{2} + \binom{n}{2}$. The desired upper bound then follows directly by applying the well-known bound that the number of regions formed by partitioning $\R^d$ by $r$ hyperplanes is at most $\sum_{i=0}^{d} \binom{r}{i}$ (see, e.g., \cite[Proposition 2.4]{StanleyIntroduction}) to Equation \ref{eqn:regions}.

This essentially gives us our desired upper bound with some further analysis. By applying this upper bound and the unimodality of binomial coefficients, we have that
\begin{align*}
    |\rSYT(m,n)| &\leq \frac{1}{n!m!} \sum_{i=0}^{n+m} \binom{2\binom{m}{2}\binom{n}{2}+\binom{m}{2} + \binom{n}{2}}{i}\\
    &\leq \frac{n+m}{n!m!} \binom{2\binom{m}{2}\binom{n}{2}+\binom{m}{2} + \binom{n}{2}}{n+m}\\
    &\leq \frac{n+m}{m!n!} \binom{\frac{1}{2}(m^2-m+1)n^{2}}{m+n}\\
    &= \frac{\prod_{i=0}^{m+n-1} (\frac{1}{2}(m^2-m+1)n^{2}-i)}{m!n!(m+n-1)!} \\
    &\leq \frac{(\frac{1}{2}(m^2-m+1)n^2)^{m+n}}{m!n!(m+n-1)!}. 
\end{align*}
Note that, for the second and third  inequalities, we use that $n 
\geq m \geq 3$, which we may assume without loss of generality. We can then use the fact that $\left(\frac{n}{e}\right)^n \leq n!$ (a version of Stirling's approximation) to obtain the following.  
\begin{align*}
    |\rSYT(m,n)|   &\leq \frac{(\frac{1}{2}(m^2-m+1)n^2)^{m+n}}{m!n!(m+n-1)!} \\
    & \leq \left(\frac{m^2-m+1}{2}\right)^{m+n}n^{2(m+n)}  \left(\frac{e}{n}\right)^n  \left(\frac{e}{m}\right)^m  \left(\frac{e}{m+n-1}\right)^{m+n-1}\\
    &\leq \left(\frac{(m^2-m+1)e^2}{2}\right)^{m+n}\frac{n^{m+n}}{(m+n-1)^{(m+n-1)}} \left(\frac{n}{m}\right)^{m} \\
    &\leq \left(\frac{(m^2-m+1)e^2}{2}\right)^{m+n} n \left(\frac{n}{m}\right)^{m}. 
\end{align*}
 
For $m = n$, this upper bound becomes $n^{4n+o(n)}$.
\end{proof}

\begin{proof}[Proof of \cref{cor::prob_to_zero}]

Let $\beta$ denote the number of $m \times n$ standard Young tableaux.  It is well-known (and follows from the hook length formula~\cite{HookLength}) that
$$\beta = (mn)! \prod_{j=0}^{m-1} \frac{j!}{(n+j)!}\,.$$
Using the bound
$$\prod_{j=0}^{m-1} \frac{j!}{(n+j)!} \geq \prod_{j=0}^{m-1} \frac{1}{(n+m)^j n!} = \left((n+m)^{\binom{m}{2}}(n!)^m\right)^{-1}\,$$
along with the fact that $\left(\frac{n}{e}\right)^{n} \leq n! \leq 2n\left(\frac{n}{e}\right)^n$ for all $n \geq 3$, we obtain
\begin{align*}
    \beta &\geq \frac{(mn)!}{(n+m)^{\binom{m}{2}} (n!)^{m}} \geq \frac{(mn/e)^{mn}}{(n+m)^{\binom{m}{2}} (2n)^m (n/e)^{mn}} = \frac{m^{mn}}{(n+m)^{\binom{m}{2}} (2n)^m} .
\end{align*}

We now employ two separate bounds to show the proportion of realizable tableaux tends to $0$, one for $m_n 
\leq 33$ and one for $m_n > 33$. Given a sequence $(m_n)_{n 
\in \mathbb{N}}$, we can split it into two subsequences, one consisting of terms at most $33$ and the remaining terms in the other, so that each bound applies.  To start, we have by Theorem \ref{thm::upper and lower},
\begin{align*}
    \frac{|\rSYT(m,n)|}{\beta} &\leq n \left(\frac{n}{m}\right)^m \left(\frac{(m^2-m+1)e^2}{2}\right)^{m+n} \left( \frac{(n+m)^{\binom{m}{2}}(2n)^m}{m^{mn}}\right) \\
    &= n \left( \frac{2n^2}{m} \right)^m  (n+m)^{\binom{m}{2}}m^{m^2} \left(\frac{(m^2-m+1)e^2}{2m^{m}}\right)^{n+m} \\
    &\leq n \left( \frac{2n^2}{m} \right)^m (2nm^2)^{m^2/2}  \left(\frac{(m^2-m+1)e^2}{2m^{m}}\right)^{n+m}.
\end{align*}
Suppose first that $m \leq 33$ (in fact, any bound on $m$ suffices here). Since $\frac{(m^2-m+1)e^2}{2m^{m}} < 1$ for $m \geq 3$, the exponential term in $n$ dominates the polynomial term in $n$, so since $m \leq 33$, this entire quantity goes to $0$ as $n$ goes to infinity.  This handles the claim for bounded $m$, so we now handle the case of large $m$.  

Suppose $m > 33$. Then, in particular, $m^3 \geq \frac{(m^2-m+1)e^2}{2}$, so we can further simplify the above upper bound as
\begin{align*}
    \frac{|\rSYT(m,n)|}{\beta} \leq n \left( \frac{2n^2}{m} \right)^m (2nm^2)^{m^2/2} m^{-(m-3)(n+m)} = \left(\frac{n^{m/2 + 1/m + 2 +\log_n (2)(m/2+1)}}{m^{n - 3n/m - 2}}\right)^m
\end{align*}

Since $n \geq m \geq 33$, we have \[\frac{m}{2} + \frac{1}{m} + 2 + \log_n(2)(m/2+1)  = m\left(1/2 + 1/m^{2} + 2/m + \frac{\ln(2)}{\ln(n)}(1/2 + 1/m)\right)\leq \frac{2m}{3}.\]
Using the fact that $n \geq m$ implies $n^m \leq m^n$, we have
\begin{align*}
\frac{n^{m/2 + 1/m + 2 +\log_n (2)(m/2+1)}}{m^{n(1-3/m) - 2}} &\leq \frac{n^{2m/3}}{m^{n - 3n/m - 2}} 
&\leq \frac{m^{2n/3}}{m^{n - 3n/m - 2}}
&= \frac{1}{m^{n/3 - 3n/m - 2}}
&\leq \frac{1}{m^{8n/33 - 2}}
&\leq \frac{1}{m^{6n/33}}\,.
\end{align*}
We can therefore conclude in the case $m > 33$ that $\frac{|\rSYT(m,n)|}{\beta} \leq m^{-6(mn)/33}$, which also decays exponentially as $n$ goes to infinity.  Combining this with the bound for $m \leq 33$ proves the claim.
\end{proof}

\subsection{Single-row extensions and a lower bound}\label{subsec:single-row-extensions}
To prove the lower bound from \cref{thm::upper and lower}, we introduce \emph{single-row extensions} of realizable rectangular tableaux and apply a recursive argument.

\begin{definition}Given a rectangular tableau $T \in \rSYT(m-1,n)$, a {\bf single-row extension} of $T$ is a rectangular tableau $T' \in \rSYT(m,n)$ such that when the bottom row of $T'$ is removed, the relative order of the remaining entries corresponds to $T$. 
\end{definition}

\begin{lem}\label{lem:min-extensions}
Any tableau $T \in \rSYT(m-1,n)$ has at least $\binom n 2 + 1$ single-row extensions. 
\end{lem}
\begin{proof}Fix increasing $x = (x_1,x_2,\dots,x_{m-1})$ and $y = (y_1,y_2,\dots,y_n)$ such that $\tabfunc(x\circ y) = T$.
Since the set of $x$ and $y$ such that $x \circ y$ realizes a particular tableau is a full dimensional region of the fan given by subdividing $\calP_{n,m}$ by $\calC_{n,m}$, we can assume that differences of the form $y_i-y_j$ for $i > j$ are all distinct, because they may be chosen generically from a full dimensional region of a hyperplane arrangement. 

Let $x_m$ be a real number greater than $x_{m-1}$, and let $x' = (x_1,x_2,\dots,x_{m-1}, x_m)  \in \mathbb{R}^m$. The relative position of $x_m - x_{m-1}$ among the positive differences $y_i - y_j$ for $i>j$ determines if the $(m-1,i)$ entry of $\tabfunc(x'\circ y)$ comes before or after the $(m,j)$ entry. Thus, each of these $\binom{n}{2}+1$ relative positions of $x_m - x_{m-1}$ yields a different extension of $x+y$ to a realizable $m \times n$ tableau. \end{proof}

We are now ready to prove the lower bound of \cref{thm::upper and lower}, while the upper bound was handled in Subsection~\ref{subsec:upper-bound-rect}.

\begin{proof}[Proof of Lower Bound in \cref{thm::upper and lower}]
By Lemma~\ref{lem:min-extensions}, every $(m-1) \times n$ tableau has at least $\binom{n}{2}+1$ single-row extensions, yielding the recursion. The lower bound in the second statement of Theorem~\ref{thm::upper and lower} follows from a repeated application of this recursion together with the observation that $|\rSYT(m,n)| = |\rSYT(n,m)|$; in particular, we have
\begin{align*} |\rSYT(n,n)| &\geq |\rSYT(n-1,n)|\left( \binom{n}{2} + 1 \right)\\
&\geq |\rSYT(n-1,n-1)| \left( \binom{n-1}{2} + 1 \right) \left( \binom{n}{2} + 1 \right) .\qedhere\end{align*}\end{proof}

While Lemma~\ref{lem:min-extensions} is sufficient for proving the lower bound in Theorem~\ref{thm::upper and lower}, it is a special case of the following more general result. 

\begin{prop}\label{prop:general-extensions}
Let $x = (x_1,\dots,x_{m-1})\in \R^{m-1}$ and $y = (y_1,\dots,y_n)\in \R^n$. Suppose that $x$ and $y$ are increasing and all $x_{i} + y_{j}-y_{k}$ are pairwise distinct across all pairs $j \neq k$, and let $T = \tabfunc(x \circ y)$. Then the size of the set
\[\left\{\tabfunc(x' \circ y) \mid x' = (x_1,\dots,x_{m-1},x_m)\text{ where } x_m > x_{m-1}\right\}\]is precisely
\[n^2(m-1) + 1 - (\text{sum of entries in the bottom row of $T$}).\]
In particular, this observation holds for a generic choice of $x$ and $y$.
\end{prop}
\begin{proof} The general method of the proof is similar to the proof of Lemma~\ref{lem:min-extensions}, but we now consider all coordinates of $x$ instead of just $x_{m-1}$. 

Each possible tableau $\tabfunc(x' \circ y)$ is in correspondence with a choice of relative position for $x_m$ among the values $x_i + y_j - y_k$ for $(i,j,k) \in [m-1]\times[n]\times[n]$. These positions are distinct by hypothesis. Thus, there are $n^2(m-1) + 1$ possible relative positions for $x_{m}$, but we must also take into consideration the assumption that $x_m > x_{m-1}$.  Then we must ignore the values of $x_i + y_j - y_k$ that are less than $x_{m-1}$, as $x_m$ cannot be chosen to be less than these values.  For each fixed $k$, this number of $x_i + y_j - y_k$ that are less than $x_{m-1}$ is equal to the $(m-1,k)$-entry of $\tabfunc(x \circ y)$.  Iterating over all values of $k \in [n]$, we see that the total number of $x_i + y_j - y_k$ that are less than $x_{m-1}$ is the sum of the entries in the bottom row of $\tabfunc(x \circ y)$.  The result then follows by subtracting this sum from the total possible positions among all $x_i + y_j - y_k$ for $(i,j,k) \in [m-1]\times[n]\times[n]$.
\end{proof}
We note that better lower bounds for small $m$ can be found by applying \cite[Theorem 7]{MallowsVanderbei}. 

It is straightforward to show that the number of extensions found in Proposition~\ref{prop:general-extensions} ranges from $\binom{n}{2}+1$, when the last row of $T$ is maximized, to $\binom{n}{2}(m-1)+1$ when the last row of $T$ is minimized. 

One limitation of Proposition~\ref{prop:general-extensions} is that it requires a specific choice of vectors $x$ and $y$. While the resulting number of tableaux does not depend on this choice, the tableaux obtained will vary. In particular, the total number of single-row extensions of $T$ may be larger than the value obtained from Proposition~\ref{prop:general-extensions}.

For an extreme example of this phenomenon, let $T$ be the $(m-1) \times n$ tableau with entries ${T_{ij} = (j-1)(m-1)+i}$, that is, the transpose of the $n \times (m-1)$ superstandard tableau. In particular, this is the unique tableau such that for $j \not= j'$, we have $T_{ij} > T_{i'j'}$ if and only if $j>j'$. Notice that $T \in \rSYT(m-1,n)$ because $T = \tabfunc(x \circ y)$ whenever $y_i - y_j > x_k - x_\ell$ for all $i> j$ and $k>l$. 

\begin{prop}\label{prop:manyextensions} Let $T$ be the $(m-1) \times n$ tableau defined in the previous paragraph. There are at least $C_n$ single-row extensions of $T$, where $C_n$ is the $n^{th}$ Catalan number. 
\end{prop}

\begin{proof}
To prove this proposition, we give an injective map from elements of $\rSYT(2,n)$ to single row extensions of $T$. The result then follows from the observation that $|\text{rSYT}(2,n)| = C_{n}$, see \cite[Theorem 7]{MallowsVanderbei}. An illustration of the method used here is shown in Figure~\ref{fig:extension}.

Fix $T' \in \rSYT(2,n)$ and choose a positive $x'\in \R$ and an increasing $y'\in \R^n$ such that $\tabfunc((0,x') \circ y') = T'$. We will now define an increasing $x''=(x''_1 = 0,x''_2,\dots,x''_{m-1}, x''_m =x') \in \R^m$ such that $\tabfunc(x''\circ y')$ is a single-row extension of $T$. 

Let $\epsilon>0$ be smaller than the minimal difference between the entries of $(0,x') \circ y'$. For $1<i<m$, let $x''_i = (i-1)\epsilon/(m-1)$. Finally, let $T'' = \tabfunc(x'' \circ y')$. See Figure~\ref{fig:extension} for an example of this extension. 

By construction, for $i,i' \in [m-1]$ and $j \not= j' \in [n]$, we have $T''_{ij} > T''_{i'j'}$ if and only if $T''_{1j} > T''_{1j'}$, which holds precisely when $j>j'$. In particular, $T''$ is a single-row extension of $T$. Furthermore, for $i \in [m-1]$ and $j \not= j' \in [n]$, it is also immediate that $T''_{mj} > T''_{ij'}$ if and only if $T'_{2j}>T'_{1j'}$. This implies that for each choice of $T'\in \rSYT(2,n)$, we obtain a distinct $T'' \in \rSYT(m,n)$. \end{proof}

\begin{figure}
\[ T = \begin{blockarray}{ccccc}
&&&&\\
\begin{block}{|c|c|c|c|c|}
\hline 1 & 5 & 9  & 13 & 17\\
\hline 2 & 6 & 10 & 14 & 18\\
\hline 3 & 7 & 11 & 15 & 19\\
\hline 4 & 8 & 12 & 16 & 20\\
\hline
\end{block}
\end{blockarray}\hspace {1 cm}\begin{blockarray}{cccccc}
& 0 & 1 & 5 & 10 & 12\\
\begin{block}{c(ccccc)}
0 & 0 & 1 & 5 & 10 & 12\\
6 & 6 & 7 & 11 & 16 & 18\\
\end{block} 
\end{blockarray}
\hspace{.2 cm}\pmb{\Rightarrow}\hspace{.2 cm} T' = 
\begin{blockarray}{ccccc}
&&&&\\
\begin{block}{|c|c|c|c|c|}
\hline 1 & 2 & 3 & 6 & 8\\
\hline 4 & 5 & 7 & 9 & 10\\
\hline
\end{block}
\end{blockarray}\]

\[\begin{blockarray}{cccccc}
& 0 & 1 & 5 & 10 & 12\\
\begin{block}{c(ccccc)}
0 & 0 & 1 & 5 & 10 & 12\\
\epsilon/4 & \epsilon/4 & 1+\epsilon/4 & 5+\epsilon/4 & 10+\epsilon/4 & 12+\epsilon/4\\
\epsilon/2 & \epsilon/2 & 1+\epsilon/2 & 5+\epsilon/2 & 10+\epsilon/2 & 12+\epsilon/2\\
3\epsilon/4 & 3\epsilon/4 & 1+3\epsilon/4 & 5+3\epsilon/4 & 10+3\epsilon/4 & 12+3\epsilon/4\\
6 & 6 & 7 & 11 & 16 & 18\\
\end{block}
\end{blockarray}
\hspace{.2 cm}\pmb{\Rightarrow}\hspace{.2 cm} T'' = 
\begin{blockarray}{ccccc}
&&&&\\
\begin{block}{|c|c|c|c|c|}
\hline 1 & 5 & 9 & 15 & 20\\
\hline 2 & 6 & 10 & 16 & 21\\
\hline 3 & 7 & 11 & 17 & 22\\
\hline 4 & 8 & 12 & 18 & 23\\
\hline 13 & 14 & 19 & 24 & 25\\
\hline
\end{block} 
\end{blockarray}
\]
 \caption{This figure shows how to construct a different single-row extension of the tableau $T$ in the upper left for each $T' \in \rSYT(2,5)$. This construction is used in the proof of Proposition~\ref{prop:manyextensions}.}\label{fig:extension}
\end{figure}

In light of Propositions \ref{prop:general-extensions} and \ref{prop:manyextensions}, it seems an approach to sorting $X+Y$ could be in adding a single column one at a time and managing to merge the two sorted lists efficiently. Clearly one can do this using $O(n)$ binary searches, which would take $O(n\log(n))$ time, but leveraging the realizability may allow one to do better. 


\section{Staircase Tableaux} \label{sec:staircase}
To bound the number of realizable staircase tableaux, we leverage their equivalence to geometrically realizable sorting networks. A \emph{sorting network} is a sequence $\sigma_{0}, \sigma_{1}, \dots, \sigma_{\binom{n}{2}}$ of permutations such that $\sigma_{0}$ is the identity, $\sigma_{\binom{n}{2}}$ is the order reversing permutation, $\sigma_{i+1} = s_{k} \sigma_{i}$ for some adjacent transposition $s_{k}$, and $\sigma_{i+1}$ has more inversions than $\sigma_{i}$ for all $0 \leq i \leq \binom{n}{2}-1$. Sorting networks correspond to staircase standard Young tableaux of length $n-1$ via the Edelman-Greene bijection~\cite{EdelmanGreene}. We are interested in a certain subclass of sorting networks for comparison. 

\begin{definition}\label{def::geom_realizable}
Consider a subset $X \subseteq \R^{2}$ with $|X| = n$ such that the slope between any pair of points is distinct and no two have the same first coordinate. Then $X$ can be totally ordered by the first coordinate, i.e., the dot product with the vector $(1,0)$. Furthermore, the set of total orderings of the point configuration induced by taking the dot product with a vector $z_{\theta} = (\cos(\theta), \sin(\theta))$ as $\theta$ ranges from $0$ to $\pi$ can be viewed a sequence of permutations by comparing with the initial total order.  Since these permutations differ by a transposition as $\theta$ increases, this yields a sorting network. A sorting network that arises from some subset $X \subseteq \R^{2}$ in this way is called {\bf geometrically realizable}. See Figure \ref{fig:wiringdiagram} for an example of a geometrically realizable sorting network together with a set of points in $\R^{2}$ that realize it.

\end{definition}

\begin{figure}
    \centering
    \[\begin{tikzpicture}
    \draw (-3,.5) node[blue] {$(0,0)$};
    \draw (-2,2.5) node[red] {$(1,2)$};
    \draw (-1,1.5) node[black] {$(2,1)$};
    \draw (-3,0) node[blue, circle,fill, inner sep = 1.5pt] {};
    \draw (-2,2) node[red, circle,fill, inner sep = 1.5pt] {};
    \draw (-1,1) node[black, circle,fill, inner sep = 1.5pt] {};
    \draw (.5,0) node[black] {$3$};
    \draw (.5,1) node[red] {$2$};
    \draw (.5,2) node[blue] {$1$};
    \draw (8.5,2) node[black] {$3$};
    \draw (8.5,1) node[red] {$2$};
    \draw (8.5,0) node[blue] {$1$};
    \draw[black, thick] (1,0) -- (2,0);
    \draw[red, thick] (1,1) -- (2,1);
    \draw[blue, thick] (1,2) -- (2,2);  
    \draw[black, thick] (2,0) -- (3,1);
    \draw[red, thick] (2,1) -- (3,0);
    \draw[blue, thick] (2,2) -- (3,2); 
    \draw[black, thick] (3,1) -- (4,1);
    \draw[red, thick] (3,0) -- (4,0);
    \draw[blue, thick] (3,2) -- (4,2); 
    \draw[black, thick] (4,1) -- (5,2);
    \draw[red, thick] (4,0) -- (5,0);
    \draw[blue, thick] (4,2) -- (5,1); 
    \draw[black, thick] (5,2) -- (6,2);
    \draw[red, thick] (5,0) -- (6,0);
    \draw[blue, thick] (5,1) -- (6,1); 
    \draw[black, thick] (6,2) -- (7,2);
    \draw[red, thick] (6,0) -- (7,1);
    \draw[blue, thick] (6,1) -- (7,0);
    \draw[black, thick] (7,2) -- (8,2);
    \draw[red, thick] (7,1) -- (8,1);
    \draw[blue, thick] (7,0) -- (8,0);
\end{tikzpicture}\]
    \caption{Sorting networks are often represented by a \emph{wiring diagram}, which records the set of swaps going from one permutation to another as crossing of wires from left to right. The figure depicts the wiring diagram for the sorting network $(1,2,3) \to (1,3,2) \to (2,3,1) \to (3,2,1)$. This sorting network is geometrically realizable with realization given by the points $(0,0), (1,2),$ and $(2,1)$ on the left side of the image. Namely, the orderings are induced by linear functions as follows: $(1,0)^{T}x \to (1,2,3)$, $(0,1)^{T}x \to (1,3,2)$, $(-1,1)^{T}x \to (2,3,1)$, and finally $(-1,0)^{T} x \to (3,2,1)$. }
    \label{fig:wiringdiagram}
\end{figure}

Let $X = \{v^{1}, v^{2}, \dots, v^{n}\} \subseteq \R^{2}$ be a set of points satisfying the conditions of \cref{def::geom_realizable}. Then each ordering is determined by whether $z_{\theta}^{T}v^{i} > z_{\theta}^{T} v^{j}$ for each $i < j \in [n]$. In particular, consider the hyperplane arrangement $H$ with normals given by $\{v^{i} - v^{j}: 1 \leq i < j \leq n\}$. Thus, the regions of this arrangement correspond to total orderings of $v^{1}, v^{2}, \dots, v^{n}$. In particular, the ordering induced by $z_{\theta}$ is determined by which region of the hyperplane arrangement $z_{\theta}$ is contained in.

Consider the polytope $Z = \sum_{i= 1}^{n}\sum_{j = 1}^{n} [v^{i} - v^{j}, v^{j} - v^{i}]$, which  is dual to the hyperplane arrangement $H$. Let $V(Z)$ denote the set of vertices of $Z$. The \emph{upper path }on $Z$ is the path traced by the set of vertices
\[\{v \in V(Z): \text{ there exists } (a,b) \in \R \times \R_{\geq 0} \text{ such that } v = \text{argmax}_{x \in Z} (a,b) \cdot x \}.\] 
Since $Z$ is dual to $H$, the upper path corresponds to the set of regions of $H$ containing a vector with a non-negative second coordinate. These are precisely the set of regions containing $z_{\theta}$ for some choice of $\theta \in [0, \pi]$. Hence, the upper path on $Z$ is determined by a corresponding geometrically realizable sorting network. We will use this observation to define a geometric space that we may use to upper bound the number of geometrically realizable sorting networks.

\begin{lem} \label{lem:hsurf-orderings}
The number of realizable staircase tableaux is at most the number of different possible total orderings of $\{(i,j): i < j\}$ across all $c, w \in \R^{n}$ induced by the relation that
\[(i,j) <_{c,w} (k, \ell)\quad\text{ if }\quad\frac{c_{i} - c_{j}}{w_{i} - w_{j}} < \frac{c_{k} - c_{\ell}}{w_{k} - w_{\ell}}.\]
In particular, $|\rSCYT(n)|$ is at most the number of regions of the arrangement of quadratic hypersurfaces $\{(w_{k} - w_{\ell})(c_{i} - c_{j}) - (w_{i}-w_{j})(c_{k}-c_{\ell}): 1 \leq i < j \leq n \text{ and } 1 \leq k < \ell \leq n\}$ containing a point such that $w_{1} < w_{2} < \dots <w_{n}$.
\end{lem}
\begin{proof}
 Define a linear map $\pi: \R^{n} \to \R^{2}$ by $\pi(e_{i}) = v^{i}$. Then $Z = \pi\left(\sum_{i=1}^{n} [e_{i} - e_{j}, e_{j} - e_{i}]\right) = \pi(\Pi_{n})$, where $\Pi_{n}$ is normally equivalent to the permutahedron. Since $\pi$ is linear, there exist $c, w \in \R^{n}$ such that $\pi(x) = (w^{T} x, c^{T}x)$. Furthermore, since $(1,0)$ imposes a total order on $X$, we may up to reordering of indices assume that $w_{1} < w_{2} < \dots < w_{n}$. The upper path then corresponds to a path from a vertex with minimal first coordinate to one with maximal first coordinate. Since $Z$ is centrally-symmetric, i.e., $Z = -Z$, this path uses all edge directions $v^{i} -v^{j}$ exactly once. The ordering of the edges that are used then determines the path. At each step, the slope of an edge in the path must decrease. Thus, the total ordering on edges in the path is determined by the total ordering on the slopes $\left\{\frac{c_{i} - c_{j}}{{w_{i} - w_{j}}}: 1 \leq j < i \leq n \right\}$. 

Hence, the set of geometrically realizable sorting networks is always given by a total ordering on slopes by some choice of $c$ and $w$. 
\end{proof}

Note that $\frac{c_{i} - c_{j}}{w_{i} - w_{j}} < \frac{c_{k} - c_{\ell}}{w_{k} - w_{\ell}}$ if and only if $(c_{i} - c_{j})(w_{k} - w_{\ell}) - (c_{k} - c_{\ell})(w_{i} - w_{j}) < 0$, assuming the coordinates of $w$ are increasing. Hence, the set of possible orderings from \cref{lem:hsurf-orderings} is the number of regions of the hypersurface arrangement with hypersurfaces
\[H_{i,j,k,\ell} = \{c, w \in \R^{n}: (c_{i} - c_{j})(w_{k} - w_{\ell}) - (c_{k} - c_{\ell})(w_{i} - w_{j}) = 0\},\]
for every $1 \leq i < j \leq n$, $1 \leq k < \ell \leq n$ and $(i,j) \neq (k,\ell)$. Furthermore, for a geometrically realizable sorting network, we may always assume that the coordinates of $w_{i}$ are increasing. Hence, to bound the number of geometrically realizable sorting networks or equivalently, the number of realizable staircase tableaux, it suffices to bound the number of regions of this hypersurface arrangement when intersected with the cone $\{(c,w): w_{i} < w_{i+1} \text{ for all } 1 \leq i \leq n\}$. 

To do this, we rely on the following well-known result.

\begin{theorem}[Milnor \cite{Milnor}, Thom \cite{Thom}, Warren \cite{Warren}] \label{thm:bigmachine}
Given a set of $N$ polynomials in $\mathbb{R}[x_{1}, \dots, x_{t}]$ of degree at most $D$, the number of different regions of the corresponding hypersurface arrangement is at most
\[\left(\frac{4eD N}{t}\right)^{t}. \]
\end{theorem}

This general upper bound for arrangements implies our asymptotic upper bound.

\begin{lem}\label{lem: staircase upper bound}
The number of realizable staircase tableaux is at most $n^{5n+o(n)}$.
\end{lem}

\begin{proof}
By Lemma \ref{lem:hsurf-orderings} and the discussion thereafter, it suffices to bound the number of regions of the hypersurface arrangement with hypersurfaces 
\[H_{i,j,k,\ell} = \{c, w \in \R^{n}: (c_{i} - c_{j})(w_{k} - w_{\ell}) - (c_{k} - c_{\ell})(w_{i} - w_{j}) = 0\},\]
where $i,j,k,\ell$ satisfy $1 \leq i < j \leq n$, $1 \leq k < \ell \leq n$, and $(i,j) \neq (k,\ell)$, intersected with the cone $\{(c,w)\in \R^{n}: {w_{i} < w_{i+1} \text{ for all } 1 \leq i \leq n}\}$. The degree $D$ of these polynomials is $2$, the number of variables is $2n$, and the number of polynomials is $\binom{\binom{n}{2}}{2} = O(n^{4})$. Thus, by \cref{thm:bigmachine}, the total number of regions of the hypersurface arrangement $\mathcal{H}$ of all $H_{i,j,k,\ell}$ is at most 
\[\left(\frac{4e\cdot 2 \cdot \binom{\binom{n}{2}}{2}}{2n}\right)^{2n} = n^{6n + o(n)}. \]
Note that this arrangement is symmetric with respect to reordering the coordinates $w_{1}, \dots, w_{n}$. Hence, the set of regions containing a point such that $w_{1} < w_{2} < \dots < w_{n}$, is $\frac{1}{n!}$ times the number of regions of $\mathcal{H}$. Therefore, the total number of realizable staircase tableaux is at most $\frac{n^{6n+o(n)}}{n!} = n^{5n + o(n)}$. 
\end{proof}

\begin{figure}
    \centering
    \[\begin{tikzpicture}[scale = .5]
    \draw[black] (-2,-1) -- (6,3);
    \draw[black] (-2, .5) -- (6, -1.5 );
    \draw[black] (-2, 5) -- (6,-3); 
    \draw[green] (-2,2) -- (6,-6);
    \draw[green] (-2,2) -- (6,0);
    \draw[green] (-2,-4) -- (6,0);
    \draw (0,0) node[red, circle,fill, inner sep = 1.5pt] {};
    \draw (2,1) node[blue, circle,fill, inner sep = 1.5pt] {};
    \draw (4,-1) node[orange, circle,fill, inner sep = 1.5pt] {};
\end{tikzpicture}
    \]
    \caption{In the proof of Lemma \ref{lem:staircaselowbound}, one considers a generic configuration of points such that the order on slopes of all lines between the points gives a realization of a given staircase tableau. An example configuration of $3$ points together with all lines of each slope drawn through each set of points is depicted in the figure. The regions of this arrangement correspond to the possible extension of this realizable tableau given by adding a new point to the configuration. }
    \label{fig:staircaseproof}
\end{figure}

It remains to find lower bounds, and to do so we apply a similar argument to the case of realizable rectangular tableaux by deriving a recurrence.

\begin{lem} \label{lem:staircaselowbound}
For every $n \ge 2$, we have that 
\[|\rSCYT(n)| \geq \left(\frac{(5n-3)(n+2)n(n-1)(n-2)}{120} \right)|\rSCYT(n-1)|.\]
\end{lem}

\begin{proof}
For each realizable staircase tableau $T$ of shape $(1,2,\dots, n-1)$, fix a realization given by a total ordering of slopes of the line segments between pairs from a list of $n$ points $(x_{1},y_{1}), \dots, (x_{n},y_{n})$. We can do this, since the set of $x, y$ that induce the same tableau is a full dimension region of a hypersurface arrangement. 

To create a new tableau, we add a new point $(x_{n+1},y_{n+1})$ to this configuration. Consider the line arrangement $\calA_n$ given by taking all lines through each of the points $(x_{i},y_{i})$ for $1 \leq i \leq n$ of slope $\frac{y_{k} - y_{j}}{x_{k} - x_{j}}$ for each $1 \leq j < k \leq n$.  Let $I_n$ be the set of points in $\R^2$ where any of the lines in $\calA_n$ intersect.  Let $L(x,y)$ be the set of lines through $(x,y)$ of slope $\frac{y_{k} - y_{j}}{x_{k} - x_{j}}$ for each $1 \leq j < k \leq n$.  The points $(x,y) \in \R^2$ such that $L(x,y)$ intersects nontrivially with $I_n$ form a union of lines in $\R^2$.  Thus, inside any region in the arrangement $\calA_n$, we can choose a point $(x_{n+1},y_{n+1})$ such that no line in $L(x_{n+1},y_{n+1})$ passes through any point of $I_n$. 

We now want to lower bound the number of new regions formed by such a choice of $(x_{n+1},y_{n+1})$. Given such a choice, consider the line arrangement $\calA_{n+1}$ and the set of intersection points $I_{n+1}$, and let $r(\mathcal{A}_{n})$ denote the number of regions of the $\mathcal{A}_{n}$.  It is a standard result in the theory of line arrangements (See for example the proof of Theorem $8.4$ in \cite{CompGeoTextbook}) that the number of new regions is at least the number of new intersection points, since segments of the new line incident to the new intersection point must subdivide existing regions.  That is, we have 
$$r(\mathcal{A}_{n+1}) - r(\mathcal{A}_n) \geq |I_{n+1} \setminus I_n|\,.$$
By our restriction on the choice of $(x_{n+1},y_{n+1})$, we know that the intersection point between a line in $L(x_{n+1},y_{n+1})$ and each non-parallel line in $\calA_n$, of which there are $\left(\binom{n}{2} - 1\right)(n-1)$, does not lie in $I_n$.  Moreover, all these intersection points are distinct since the lines in $L(x_{n+1},y_{n+1})$ all pass through $(x_{n+1},y_{n+1})$, and hence cannot intersect elsewhere.  Thus, we have
$$|I_{n+1} \setminus I_n| \geq |L(x_{n+1},y_{n+1})| \cdot \left(\binom{n}{2} - 1\right)(n-1) = \binom{n}{2}\left(\binom{n}{2} - 1\right)(n-1) = 2(n-1) \binom{\binom{n}{2}}{2}\,.$$

Applying this bound inductively, we may lower bound the number of regions of this line arrangement to be at least $\sum_{i=1}^{n} 2(i-1) \binom{\binom{i}{2}}{2}.$ Thus, the number of possible extensions of any fixed $T$ by adding a new point is at least $\sum_{i=1}^{n} 2(i-1) \binom{\binom{i}{2}}{2} $. Note that this map is at most $(n+1)$-to-$1$, since the original tableau may be recovered by deleting one of the points from the new configuration. Therefore, the total number of new tableaux generated across all realizable staircase tableaux of shape $(1,2,\dots,n-1)$ is at least 
\[\frac{1}{n+1}\sum_{i=1}^{n} 2(i-1) \binom{\binom{i}{2}}{2} |\rSCYT(n-1)| = \frac{(5n-3)(n+2)n(n-1)(n-2)}{120} |\rSCYT(n-1)|, \]
where this final equality may be verified using a computer algebra system such as Mathematica.
\end{proof}

\begin{remark}
We are thankful to an anonymous referee for identifying a closed form expression for Lemma~\ref{lem:staircaselowbound}.  
\end{remark}

\cref{thm:staircasebounds} then follows from the upper bound of $n^{5n+o(n)}$ we already showed in Lemma \ref{lem: staircase upper bound} and induction on the result of \cref{lem:staircaselowbound} to find a lower bound of $n^{5n-o(n)}$.

\begin{remark}\label{rem:realizablestaircase}
    Goodman and Pollack's 1986 work \cite{GoodmanPollackBounds} proves the weaker upper bound $n^{8n+o(n)}$ on the number of allowable sequences (equivalently, $|\rSCYT(n)|$).  A straightforward application of the hook length formula \cite{HookLength} shows that the total number of staircase standard Young tableaux is 
    $$\frac{\binom{n+1}{2}!}{\prod_{k=0}^{n-1} (2k+1)^{n-k}} \leq \frac{\binom{n+1}{2}!}{(2n)^{\binom{n+1}{2}}} = n^{n^2/2 + o(n^2)}\,.$$  Thus, Goodman and Pollack's result is enough to show that the proportion of staircase standard Young tableaux of shape $(n,n-1,\dots,1)$ that are realizable is $n^{-n^2/2 + o(n^2)}$.  In the language of Angel, Gorin, and Holroyd \cite{PatternTheorem}, the staircase standard Young tableaux are equivalent to sorting networks, and those tableaux that are realizable are precisely the geometrically realizable sorting networks.  Thus, Goodman and Pollack's earlier work provides an alternate proof to Angel, Gorin, and Holroyd's result \cite[Theorem 1.3]{PatternTheorem} that the proportion of sorting networks that are geometrically realizable tends to $0$ and $n$ tends to $\infty$, and in fact gives a stronger bound on the rate of decay (answering a follow-up question of Angel, Gorin, and Holroyd).
\end{remark}

\section{Further Directions}
\label{sec:FurDir}

In the work of Mallows and Vanderbei, they characterize the non-realizable $3\times n$ tableaux for $3 \leq n\leq 6$ as well as the non-realizable $4\times4$ tableaux in terms of \emph{taboo configurations}~\cite{MallowsVanderbei}.  These are minimal sets of inequalities between entries of a tableau that guarantee that the tableau is not realizable.  Farkas' lemma guarantees that for every non-realizable tableau, there is some minimal set of inequalities between linear combinations of the entries that demonstrates non-realizability.  However, in the taboo configurations provided by Vanderbei and Mallows, it is enough to consider inequalities between two entries, with each entry appearing in at most one such inequality.  We conjecture that this is always the case.

\begin{conjecture}
An $m \times n$ tableau $M = \begin{bmatrix}M_x\end{bmatrix}_{x \in [m]\times[n]}$ is not realizable if and only if there exist two disjoint sets of equal size $A = \{a_1,\dots,a_\ell\},B=\{b_1,\dots,b_\ell\} \subseteq [m]\times[n]$ such that
\begin{enumerate}[(i)]
\item\label{cond::real1} $M_{a_k} < M_{b_k}$ for each $k \in \{1,\dots,\ell\}$,
\item\label{cond::real2} $\left|A \cap \left(\{i\} \times[n]\right)\right| = \left|B \cap \left(\{i\} \times[n]\right)\right|$ for each $i \in [m]$, and
\item\label{cond::real3} $\left|A \cap \left([m]\times\{j\}\right)\right| = \left|B \cap \left([m]\times\{j\}\right)\right|$ for each $j \in [n]$.
\end{enumerate}
\end{conjecture}

It is straightforward to prove that the above conditions are sufficient to show that a tableau is not realizable (see Proposition \ref{prop:tabooconfig}). Thus, the difficulty lies in showing whether these conditions are necessary for non-realizability.

In Subsection~\ref{subsec:single-row-extensions}, we provide an asymptotically tight lower bound on the number of realizable rectangular tableaux by lower bounding the number of single-row extensions of a fixed realizable rectangular tableau.  In Proposition \ref{prop:general-extensions}, we give an exact formula for the number of single-row extensions of a realizable rectangular tableau $\tabfunc(x \circ y)$, where $x$ and $y$ are fixed, that can be obtained appending an entry to $x$.  This formula only depends on the size of the tableau and the sum of the entries in the last row.  It would be interesting to perform a similar enumeration without fixing the vectors realizing the tableau.

\begin{op}\label{op:formula-single-row-exts} Find a recursive formula for $\rSYT(m,n)$ (possibly in terms of single-row extensions) and explore the generating functions $R_m(z) \colonequals \sum_{n} \rSYT(m,n)z^n$.
\end{op}

We can generalize the notion of realizability to higher-dimensional Young tableaux and to functions other than outer sums.  Consider a $d$-dimensional array having length $n$ in each dimension with entries $a_{i_1,\dots,i_d}$, along with a degree $D$ polynomial $f \in \R_{\geq 0}[z_1,\dots,z_d]$.  Given $x_j \in \R^{n}$ for $j = 1,\dots,d$, we define the entries of the array by $a_{i_1,\dots,i_d} = f(x_{1,i_1},\dots,x_{d,i_d})$.  When all the entries are pairwise different and the entries are increasing within each vector $x_j$, we can then obtain a $d$-dimensional standard Young tableau by replacing each entry with its relative position among the entries of the array, using the numbers $1$ through $n^d$.  We call such a tableau an \emph{$f$-realizable $n^{\times d}$ (standard Young) tableau}.  Taking $d = 2$ and $f = z_1 + z_2$ recovers the realizable $n\times n$ tableaux.  As a natural analogue to Corollary \ref{cor::prob_to_zero}, we can quickly show that almost all $n^{\times d}$ standard Young tableaux are not realizable.

\begin{prop}
For $d \geq 2$ and a polynomial $f \in \R_{\geq 0}[z_1,\dots,z_d]$, the probability that an $n^{\times d}$ tableau is $f$-realizable goes to $0$ as $n$ tends to infinity.  
\end{prop}
\begin{proof}
By a straightforward adaptation of the methods in Subsection~\ref{subsec:upper-bound-rect}, we can biject the $f$-realizable $n^{\times d}$ tableaux with the regions of an arrangement formed by $\binom{\binom{n}{d}}{2} \leq n^{2d}$ hyperplanes.  By the well-known \cref{thm:bigmachine} (due to Milnor, Thom, and Warren), the number of regions in this arrangement is at most
$$\left( \frac{4eDn^{2d}}{dn}\right)^{dn} = n^{(2d^2-d)n + o(n)}\,.$$
We now use a rudimentary lower bound for the total number of $n^{\times d}$ tableaux given by enumerating just the $n\times n$ subtableaux.  Using the hook length formula \cite{HookLength} and a similar sequence of bounds as in the proof of \cref{cor::prob_to_zero}, the number of $n\times n$ standard Young tableaux given by 
\begin{align*}(n^2)! \prod_{j=0}^{n-1} \frac{j!}{(n+j)!} &\geq \frac{(n^2)!}{(2n)^{{n \choose 2}} (n!)^n} \geq \frac{n^{n^2}}{(2n)^{{n \choose 2} + n}}
\end{align*}
We hence can see that there are at least $n^{\Omega(n^2)}$ total $n^{\times d}$ standard Young tableaux. 
\end{proof}

As the above proposition demonstrates, some of our methods extend to the generalized notion of $f$-realizable $n^{[d]}$ tableaux.  It would be interesting to obtain similar results about the asymptotic enumeration of realizable tableaux in this more general setting, possibly by using an adaptation of the existing techniques.

\begin{op}
Prove tight asymptotic bounds on the number of $f$-realizable $n^{[d]}$ tableaux.
\end{op}

Our notion of realizability for rectangular tableaux also has further generalizations. Namely, the set of matrices that arise from an outer sum are precisely the set of matrices of tropical rank $1$ (see \cite{troprank}). Our question of which tableaux are realizable is equivalent to asking which regions of the intersection of the tropical variety of rank $1$ matrices with the braid arrangement are nonempty. In fact, one can see this as the dual expression of the zonotope corresponding to our hyperplane as a projection of the permutahedron. Such a projection always exists for monotone path polytopes of zonotopes \cite{sweeppolys}. A natural extension of our question is then to ask what occurs when we allow the rank to grow. Of course for matrices of full rank, all possible total orders of coordinates are attained, but what is the smallest rank for which that phenomenon occurs?

\begin{op}
We call a rectangular tableau $r$-realizable if there exists a rank $r$ matrix such that the total ordering of the entries of that matrix corresponds to the total ordering of the entries of the tableau. In terms of $m$ and $n$, what is the minimal (tropical) rank $r(m,n)$ such that all rectangular tableaux of shape $(m,n)$ are $r(m,n)$-realizable? 
\end{op}

Note that this question may depend on the chosen notion of tropical rank, since there are many that do not coincide in general tropical geometry but do coincide in the case of rank $1$. Furthermore, considering classical rank for matrices with positive entries would still be a proper generalization of what we do here, since applying log to each entry is an order preserving bijection taking rank $1$ matrices with positive entries to tropical rank $1$ matrices. To be formal, we have the following corollary to \cref{thm::upper and lower}:

\begin{cor}
The number of possible orderings of coordinates of a rank $1$ $m \times n$ matrix with positive entries is asymptotically $m^{(2+o(1))(n+m_n)}(m_n)!n!$, for $3 \leq m_n\leq n$ and both tending to infinity.
\end{cor}

\begin{proof}
By Theorem \ref{thm::upper and lower}, the number of $m_{n} \times n$ realizable standard Young tableaux satisfy the asymptotic bound of $m^{(2+o(1))(n+m_{n})}$. These are precisely the number of possible total orderings of a matrix arising from the outer sum of two strictly increasing vectors. The total number of orderings of a tropical rank $1$ matrix is the set of orderings of the outer sum of any pair of vectors. These are precisely the set of all regions of $C_{n,m}$ and hence, by the proof of Theorem \ref{thm::upper and lower}, the asymptotic number of them is $m^{(2+o(1))(n+m_{n})} (m_{n})!n!$.

Thus, it suffices to show that the set of possible orderings of the coordinates of a rank $1$ $m \times n$ matrix with positive entries is the same as the number of orderings of a tropical rank $1$ matrix to prove the corollary. A rank $1$ matrix is of the form $x y^{T} = (x_{i}y_{j})_{i \in [m], j \in [n]}$. Suppose that $xy^{T}$ is positive, so $x_{i} y_{j} > 0$ for all $i \in [m]$ and $j \in [n]$. Then without loss of generality, we may assume that $x_{i}$ and $y_{j}$ are also positive for all $i \in [m]$ and $j \in [n]$. Apply $\log$ to each coordinate to arrive at the matrix 
\[M = (\log(x_{i}y_{j}))_{i \in [m], j \in [n]} = (\log(x_{i}) + \log(y_{j}))_{i \in [m], j \in [n]}. \]
Since $\log$ is increasing, the total ordering on the coordinates of $M$ is the same as the total ordering on the coordinates $x y^{T}$. Furthermore, $M$ is tropical rank $1$, since it is the outer sum of $(\log(x_{1}), \dots, \log(x_{m}))$ and $(\log(y_{1}), \dots, \log(y_{n}))$. Therefore, the set of orderings of entries of rank $1$ matrices with positive entries is a subset of the set of orderings of tropical rank $1$ matrices.

Similarly, a tropical rank $1$ matrix has a representation as an outer sum of not necessarily non-negative vectors $w \in \mathbb{R}^{m}$ and $z \in \mathbb{R}^{n}$. Applying the exponential function to each coordinate yields a rank $1$ matrix with positive coordinates given by the outer product of $(e^{w_{1}}, \dots, e^{w_{m}})$ and $(e^{z_{1}},\dots,e^{z_{n}})$. Furthermore, since the exponential function is increasing, the resulting matrix has the same ordering on the coordinates. Hence, the set of orderings of entries of rank $1$ matrices with positive entries is also a superset of the set of orderings of tropical rank $1$ matrices and therefore the same as desired.
\end{proof}

Another direction forward comes from the sorting networks perspective on our results. One advantage of the results of Angel, Gorin, and Holroyd in \cite{PatternTheorem} for finding upper bounds on the number of realizable staircase tableaux is that they apply to other purely combinatorial notions of realizability. Namely, they show that if a single constant-size subconfiguration is forbidden for some subset of all staircase tableaux, then the probability a random staircase tableau of size $n$ is in that subset tends to $0$ exponentially fast as $n$ approaches $\infty$. It would be interesting to have an analogous result for rectangular or square tableaux. Our results suggest that such a result could exist even for $3 \times n$ tableaux.

Similarly, geometrically realizable sorting networks play a vital role in the analysis of random staircase tableaux. In particular, Angel, Holroyd, Romik, and Virag observe in Theorem $5$ of \cite{RandSortNets} that a random sorting network is approximated arbitrarily well by a geometrically realizable sorting network, and Dauvergne sharpened this result in Theorem $4$ of \cite{ArchLimit}. Romik and Pittel proved an analogous limit shape result for rectangular tableaux in \cite{RectTableauxLimShape}, and we conjecture that through unpacking their work, there should be an analogous statement saying that a random rectangular tableau may be approximated by a realizable rectangular tableau for the notion of realizability we study here. 




\section*{Acknowledgments}  
This work was completed in part at the 2022 Graduate Research Workshop in Combinatorics, which was supported in part by NSF grant $\#$1953985, and a generous award from the Combinatorics Foundation. We would like to thank Shiliang Gao and other attendees of the GRWC, as well as Bernd G\"{a}rtner, Andrey Kupavskii, Igor Pak, and Dan Romik for useful discussions. The first author was supported by UIUC Campus Research Board RB 22000, the sixth author was supported by NSF grant DMS-2039316, and second, third, and fifth authors were all supported by the NSF GRFP.  We thank an anonymous referee for pointing out an error in the original proof of \cref{cor::prob_to_zero}.
\bibliographystyle{amsplain}
\bibliography{bibliography.bib} 

	
\end{document}